\documentclass[12pt,leqno]{article}
\catcode`\@=11 \long\def\@makecaption#1#2{
 \vskip 10pt
 \setbox\@tempboxa\hbox{\bf #1: \sf #2}
 \ifdim \wd\@tempboxa >\hsize \bf #1: \sf #2\par \else \hbox
to\hsize{\hfil\box\@tempboxa\hfil}
 \fi}
\catcode`\@=12
\newfont{\Bbb}{msbm10 scaled\magstep1}
\newcommand{\R}{\mbox{\Bbb R}}

\newcommand{\supp}{{\rm supp}}

\newcommand{\eps}{\varepsilon}

\newtheorem{theoo}{Theorem}%[section]
\newenvironment{theo}
  {\begin{theoo}\hspace{-5pt}{\bf .}}{\end{theoo}}
\newtheorem{lemm}[theoo]{Lemma}
\newenvironment{lem}
  {\begin{lemm}\hspace{-5pt}{\bf .}}{\end{lemm}}
\newtheorem{corr}[theoo]{Corollary}
\newenvironment{cor}
  {\begin{corr}\hspace{-5pt}{\bf .}}{\end{corr}}
\newtheorem{prr}[theoo]{Proposition}

\newtheorem{cll}[theoo]{Claim}

\newenvironment{proof}[1][Proof]{\textbf{#1.} }{\mbox{}\hfill \rule{0.5em}{0.5em}}

\begin{document}
\title{The GL-l.u.st.\ constant and asymmetry of the Kalton-Peck twisted sum in finite dimensions
\footnotetext{2010 Mathematics Subject Classification 46B20, 46B07.}
}
\author{Y. Gordon\thanks{Supported in part by the France-Israel Research Network Program
in Mathematics contract \#3-4301.}, M. Junge\thanks{Supported in part by NSF grant
DMS-0901457.}, M. Meyer$\mbox{}^*$ and S. Reisner$\mbox{}^*$}

\date{}
\maketitle

%\vspace{10mm}

\begin{abstract}
\noindent
We prove that the Kalton-Peck twisted sum $Z_2^n$ of  $n$-dimensional Hilbert spaces has GL-l.u.st.\
constant of order $\log n$ and bounded GL constant. This is the first concrete example which shows
different explicit orders of growth in the GL and GL-l.u.st.\ constants.
We discuss also  the asymmetry constants of $Z_2^n$.
\end{abstract}
\noindent

\section{Introduction}

Local unconditional structure, or l.u.s.t., is an important notion in the study of the geometry of  Banach
spaces (see for instance \cite{GL1}, \cite{KT}, \cite {PW}). The variant of l.u.st.\ that we investigate here
was introduced by Gordon and Lewis \cite{GL1} and is frequently referred to as {\em GL-l.u.st.}, another,
formally more restrictive, notion of l.u.st.\ was introduced in \cite{FJT}.

This notion was not, however, fully studied in the
finite dimensional case, where the asymptotic values of the various constants that it involves are not yet
fully understood. We show in this paper  that in the finite dimensional setting of $n$-dimensional
normed spaces, the GL constant and the l.u.s.t.\ constant can be of significantly different orders of magnitude,
by considering the $2n$-dimensional Kalton-Peck twisted sum of $n$-dimensional Hilbert spaces.
\vskip 2mm

Let us now recall some definitions (\cite{GL1}). A basis $B=(b_i)_{i\in I}$ for a Banach space $E$ is called
{\it unconditional} if there is a constant $C>0$ such that for every $x\in E$, $x=\sum_{i\in I} \xi_i b_i$
and every choice of signs $\varepsilon_i= \pm 1$, with $\varepsilon_i=1$ for all but a finite number
of $i\in I$, one has
\[||\sum_{i\in I} \varepsilon_i \xi_i b_i||\le C||x||.\]
The smallest $C$ satisfying this is called {\it the unconditional constant of $B$} and denoted by
$\chi(B)$. The {\it unconditional constant of $E$} is
\[\chi(E):= \inf\{ \chi(B);\ B\hbox{ {\rm is a basis of }}E\}.\]
More generally, define the local unconditional structure constant of $E$  (see \cite{GL1}),  $\chi_u(E)$, by
\[\chi_u(E):=\sup_{F\subset E} \big(\inf_{A,B, U} ||B|| \chi(U) ||A||\big)\]
where the $\sup$ is taken over all finite dimensional subspaces $F$ of $E$ and the infimum over all Banach spaces $U$ and all continuous operators
$A:F\to U$ and $B:U\to E$ such that $BA=I_F$, the identity on $F$. Clearly $\chi_u(E)\le \chi(E)$.
By \cite{FJT}, $\chi_u(E)=\chi_u(E')$ and $\chi_u(E)$ is finite iff $E''$ is isomorphic to a complemented subspace of a Banach lattice.
\vskip 1mm\noindent

Given a subset $\{x_i; i\in I\}$ of $E$, we denote
\[\eps_1(\{x_i; i\in I\}):= \sup\{  \sum_{i\in I} |x' (x_i)|; x'\in E', ||x'||\le 1  \}.\]

Let $E$ and $F$ be Banach spaces.The {\it weakly nuclear norm} of a linear  operator $u:E\to F$, which has a representation
$u=\sum_{i\in I} x_i' \otimes y_i$, which converges unconditionally in $L(E,F)$, is defined by
\[\eta(u)=\inf \eps_1(\{  x_i' \otimes y_i, i\in I\})\]
where the infimum is taken over all such representations of $u$.

\vskip 1mm\noindent
By Proposition 1.2. of [GL2], one has for all $u\in L(E,F)$
\[\eta(u)= \inf ||A|| \chi(U)||B||\]
where the infimum is taken over all factorizations $u=BA$, with $A:E\to U$ compact, $B:U\to F$, and $U$
is a Banach space with an unconditional basis. Hence $\eta(I_E)=\chi_u(E)$.
%%%%%%%%%%%As it is shown in ???? ,%%%%%%%%%%
It is easy to show, using e.g. techniques from \cite{R}, that
when $E=F$ is finite dimensional and $u=I_E$, one may %%%%limit%%%%%
restrict
this infimum to spaces $U$ satisfying
$\chi(U)=1$ and $A:E\to U$ to be an isometric embedding.
\vskip 2mm

The {\it Gordon-Lewis constant} of $E$, denoted by ${\rm gl}(E)$, is defined to be
\[{\rm gl}(E)=\inf\{c>0;\, \gamma_1(A)\le c\Pi_1(A); \hbox{ $F$ Banach space and $A: E\to F$}  \}.\]
Here
\[\gamma_1 (A)= \inf  \{ ||\alpha||\ ||\beta ||\ ; \beta \alpha=i_F A \}\]
where the infimum ranges over all Banach spaces $F$ and all $\alpha: E\to L^1(\mu)$,
$\beta : L^1(\mu)\to F''$,  $i_F:F\to F''$ is the canonical inclusion, and
\[\Pi_1 (A)=\inf\left\{C>0; \sum_{i=1}^n  ||Ax_i|| \le C\sup _{\pm} ||\sum_{i=1}^n \pm
\ x_i|| \hbox{ for every $n\ge 1$ and $x_1,\dots, x_n\in E$  } \right\}\]
is the classical {\it $1$-absolutely summing norm} of $A$.
It was proved in [GL] that ${\rm gl}(E)\le \chi_u(E)$ and it is well known that there exist
infinite dimensional separable Banach spaces
$E$ such that ${\rm  gl}(E)$ is finite and $\chi_u(E)$ is infinite
(the first such example is presented in \cite{JLS}).
In particular, it follows that there exists an increasing sequence
of $2n$-dimensional Banach spaces $E_n$ for which
$\chi_u(E_n)\to \infty$ and $({\rm gl}(E_n)\big)_n$ is bounded.
\vskip 2mm

We prove in the sequel that, for some constant $c>0$, one has
\[c_n=\sup\left\{\frac{\chi_u(E_n)}{{\rm gl}(E_n)}; E_n
\hbox{ is an $n$-dimensional normed space}\right\} \ge  c \log n\,.\]
To do this, we refine the proof, given by [JLS], that the Kalton-Peck space $Z_2$ (see [KP])
has bounded ${\rm gl}$-constant but
that $\chi_u(Z_2)$ is infinite.
%%%%%%This provided a concrete example
%%%%%%of a sequence of $n$-dimensional normed spaces having GL-l.u.st. but
%%%%%not l.u.s.t..
%%%%%%%%%It is clear that $c_n\le \sqrt{n}$;  we prove here that $c_n\ge c\log(n)$.
The spaces $E_n$ mentioned above are the subspaces $Z_2^n$ of $Z_2$ spanned by first $2n$
coordinates. Note that always $c_n\leq \sqrt{n}$, an interesting problem is: how big can
$c_n$ be? For example, is $c_n\geq cn^{\alpha}$ for some $\alpha>0$ and $c$ an absolute positive
constant? Perhaps this is true even with $\alpha=1/2$.

$Z_2$ (and $Z_p$) has unconditional finite dimensional decomposition into $2$-dimensional
subspaces ($2$-UFDD), many of the concrete examples of spaces with GL and without l.u.st.\
poses such structure (see \cite{KT} for references). The paper \cite{CK} presents a general
treatment of spaces having uniform UFDD (that is, UFDD for which the dimensions of the blocks are
uniformly bounded), it is shown there that such a (infinite dimensional) space either has an
unconditional basis or fails to have l.u.st. Observing the computations made in the works mentioned above,
it is plausible that if $c_n\geq cn^{\alpha}$ for some $\alpha>0$ then the examples showing this
would not be with uniform UFDD. In this respect we may quote here a conjecture that Nigel Kalton
sent us, together with other helpful suggestions, in response to a preprint version of the present paper:
\vskip 1mm

\noindent
{\bf Conjecture.} {\em If we restrict the definition of $c_n$ to $kn$-dimensional initial blocks
of spaces with a $k$-UFDD, then $c_n$ is equivalent to $(\log n)^{k-1}$.\/}
\vskip 2mm

The concept of asymmetry of an $n$-dimensional Banach space was introduced in \cite{GG} and
generalized the notion of the asymmetry of a basis. This was followed up also in \cite{GL1} as
well. We study here related notions of asymmetry, and apply them to $Z_2^n$.

For general terminology concerning the geometry of Banach spaces we refer the reader to
\cite{LT1, LT2}. Terminology concerning normed ideals of operators may be found in \cite{P} and
\cite{T}.

Finally, we wish to thank Alexander Litvak for very helpful discussions.

\section{The l.u.st.\ constant of $Z_2^n$}
Let $Z_2^n$ be the $2n$-dimensional (real) Banach space which is the
subspace of the Kalton-Peck \cite{KP} space $Z_2$, spanned by the first $2n$
coordinates.
More precisely, for $a,b\in \R^n$, $a=(a_1,\dots,a_n)$,  $b=(b_1,\dots,b_n)\in
\R^n$,  we define
\[||(a,b)||_{Z_2^n}= \|(a_1,\dots,a_n, b_1,\dots,b_n)\|_{Z_2^n}=||\sum_{j=1}^n (a_je_j +b_j f_j)||_{Z_2^n}\]
\[= \left(\sum_{j=1}^n b_j^2\right)^{1/2}+
\left(\sum_{j=1}^n\left(a_j-b_j\log\left(|b_j|\left(\sum_{i=1}^n b_i^2
\right)^{-1/2}\right)\right)^2\right)^{1/2}\,. \]
As it is proved in \cite{KP}, this quasi-norm on $\R^{2n}$ is (uniformly in $n$)
equivalent to a norm.

\vspace{1mm}\noindent
Like in \cite{JLS} we consider $Z_2^n$ as an unconditional sum
\begin{equation}
\label{eq-B}
Z_2^n=\sum_{j=1}^n \bigoplus E_j
\end{equation}
where $E_j=[e_j,f_j]$.

\vspace{1mm}\noindent
It was  proved in \cite{JLS} that $Z_2$ fails to
have l.u.st.\ They observed that  if a Banach space $E$ is of the form
$E=\oplus_{j=1}^{\infty} E_j$, a $1$-unconditional sum of finite dimensional subspaces,
then
\[{\rm gl}(E)\le \sup_{j} {\rm dim}(E_j),\]
and thus  $gl(Z_2)\le 2c$.
Hence it follows that $ gl(Z_2^n)\le 2c$, but
that the l.u.st.\ constants of $Z_2^n$ tend to infinity with $n$. We
establish here the order of growth of these constants.

\begin{theo}
\label{th-A}
One has $gl(Z_2^n)\le 2c$, and
the l.u.st.\ constant $\chi_u(Z_2^n)$ of $Z_2^n$ satisfies
\begin{equation}
\label{eq-A}
 \chi_u(Z_2^n)\sim {\log n}
\end{equation}

\end{theo}
\vspace{2mm}

Throughout the proof the letters $C$, $c$, $c_1,c_2,\ldots$ etc. will denote absolute constants
which do not depend on $n$. The same letter $c$ (etc.) may denote different
constants in different lines.
\vspace{2mm}

It is shown in \cite{KP} that the Banach-Mazur distance $d(Z_2^n,\ell_2^{2n})$ from
$Z_2^n$ to $\ell_2^{2n}$ is of the order of $\log n$.
Using  Proposition~2  of \cite{JLS}, there exists a Banach space $Y_n$ with a $1$-unconditional
basis $\{y_{j,i},j=1,\ldots,n,\hspace{1em}i=1,\ldots,k_j\}$, such that

\vskip 1mm
- $Z_2^n$ is a
subspace of $Y_n$ and, for each $1\leq j\leq n$, $E_j\subset [\{y_{j,i}\},\,\,i=1\ldots,k_j]$.

\vskip 1mm
- There exists a projection $P_n:Y_n\rightarrow Z_2^n$ such that
$P_n([\{y_{j,i},\,\,i=1\ldots,k_j\}])=E_j$  for all $j$, $1\le j\le n$. ($[A]$ denotes the
span of $A$).

\vskip 1mm
- $\|P_n\|=K_n\leq c\chi_u(Z_2^n)$.
\vskip 1mm\noindent
To justify this, one first shows (see \cite{FJT})  that  for any Banach space $E$,
\[\chi_u(E)=\inf\ ||P||_{L\to E} \]
where the infimum is taken over all Banach lattices $L$ such that $E$ is isometrically embedded in
$L$ and all projections $P:L\to E$. Then observing that $E= Z_2^n$ have bounded cotype-$q$ constants
for a fixed $q<\infty$, we use  \cite{R} to reduce
to the case when moreover $L$ is supposed to have the same bound on its cotype-$q$ constant.
Then, we follow the lines of Proposition~2  of \cite{JLS}, using
Rademacher embedding.

\vskip 1mm\noindent
We shall show
that $K_n\geq c\log n$, thus proving (\ref{eq-A}) because it is clear that
$K_n\le d(Z_2^n,\ell_2^{2n})\sim \log n$.
\vspace{3mm}

\noindent
{\bf Notation.} Let $T\,:\,Z_2^n\rightarrow Z_2^n$ be a linear operator that satisfies $T(E_j)\subset E_j$ for
$1\leq j\leq n$. We say that $T$ {\em splits through} $\{E_j\}$ or, in short, that  $T$ {\em splits}.
If $T$ splits,  the matrices representing $T|_{E_j}$ in the basis $\{e_j,f_j\}$, $1\le j\le n$, will be denoted
by
\[\left(\begin{array}{ll} \alpha_j&\beta_j\\ \gamma_j& \delta_j \end{array}\right)\,.\]
If $a=(a_1,\dots, a_n)\in \R^n$, the {\it support} of $a$ is $\supp(a):=\{j: 1\le j\le n,\ a_j\not=0\}$.
If $a,b\in \R^n$, $a=(a_1,\dots, a_n)$, $b=(a_1,\dots, b_n)$, we denote
$(a,b)=\sum_1^n (a_je_j +b_j f_j)\in Z_2^n$. We define also a vector $ab\in \R^n$ by
$ab=(a_1b_1,\dots, a_nb_n)$.
\begin{lem}
\label{lem-B}
For any $T$ that splits we have
\[\max_j\max\{|\alpha_j|,|\beta_j|,|\gamma_j|,|\delta_j|\}\leq \|T\|\,.\]
\end{lem}
\begin{proof}
If  $x=e_j$,  $1\leq j\leq n$.
then $\|Tx\|=|\gamma_{j}|+|\alpha_{j}|\leq \|T\|\|x\|=\|T\|$.
If $x=f_j$,  one gets similarly
$|\delta_{j}|+|\beta_{j}|\leq \|T\|$.
\end{proof}

\begin{lem}
\label{lem-C} With the preceding notation, there exists $C>0$ such that for every $T:Z_2^n\to Z_2^n$ that splits  and  for all $(a,b)\in \R^{2n}$, we have
\[\left(\sum_{j=1}^n \left(\delta_j a_j-\alpha_j a_j+
\gamma_j a_j\log \frac{|a_j|}{\|a\|_2}\right)^2\right)^{\frac{1}{2}}\leq
C\left(\|T\|\|(a,b)\|_{Z_2^n}+(\max_{j\in {\supp(a)}}\!|\gamma_j|)\|a\|_2\right)\,.\]
\end{lem}
\begin{proof}
It is established  in \cite{KP} that the function $F\,:\,\ell_2^n\rightarrow \ell_2^n$ given by
\[F(b)=\left(b_1 \log \frac{|b_1|}{\|b\|_2},\dots, b_n \log \frac{|b_n|}{\|b\|_2}\right)\mbox{\ for $b=(b_1,\dots, b_n)\in \R^n$}\]
is {\em quasi-linear\/} in the sense  that for every $a,b\in \R^n$
\begin{equation}
\label{eq-D}
\|F(a+b)-F(a)-F(b)\|_2\leq C\left(\|a\|_2+\|b\|_2\right)\,.
\end{equation}
We have
\[||(\gamma F(a), \gamma a)||_{Z_2^n}= ||\gamma a ||_2 +||\gamma F(a)-F(\gamma a)||_2\,,\]
so that
\[||\gamma F(a)-F(\gamma a)||_2 \leq ||(\gamma F(a), \gamma a)||_{Z_2^n}\,. \]
Hence
\begin{equation}
\label{eq-C}
||(\delta-\alpha)a+\gamma F(a)||_2\leq ||(\delta-\alpha)a+F(\gamma a)||_2+||(\gamma F(a), \gamma a)||_{Z_2^n}.
\end{equation}

\noindent
Also, unconditionality of the sum (\ref{eq-B}) implies for every $\theta=(\theta_1,\dots,\theta_n) \in \R^n$,
\begin{equation}
\label{eq-E}
\|(\theta a,\theta b)\|_{Z_2^n}\leq C\max_j|\theta_j|\  \|(a,b)\|_{Z_2^n}\,.
\end{equation}

\noindent
Defining $\alpha=(\alpha_1,\dots, \alpha_n)\in \R^b$, $\beta$, $\gamma$ and $\delta$ related to $T$ as
above, one has
\[||(a,b)||_{Z_2^n}= ||b||_2+ ||a-F(b)||_2\,, \]
\[||T(a,b)||_{Z_2^n}=||\gamma a+\delta b||_2+||\alpha a+\beta b- F(\gamma a+\delta b)||_2 \]
and
\begin{equation}
\label{eq-F}
\|F(\gamma a+\delta b)-\alpha a\|_2\leq 2\|T\| \|(a,b)\|_{Z_2^n},
\end{equation}
because
\[\|F(\gamma a+\delta b)-\alpha a\|_2\leq \|\beta b\|_2+\|\alpha a+\beta b-F(\gamma a+\delta b)\|_2
+\|\gamma a+\delta b\|_2\]
\[\leq\|T(a,b)\|_{Z_2^n}+\max_j |\beta_j|\|b\|_2\leq 2\|T\| \|(a,b)\|_{Z_2^n}\,.\]

\noindent
Applying Lemma~\ref{lem-B}, (\ref{eq-D}), (\ref{eq-E}), (\ref{eq-F}) and the fact that
$\|(F(a),a)\|_{Z_2^n}=\|a\|_2$, we get from (\ref{eq-C})
\[||(\delta-\alpha)a+\gamma F(a)||_2\leq||(\delta-\alpha)a+ F(\gamma a)||_2+ ||(\gamma F(a),\gamma a)||_{Z_2^n}
\]
\[\leq\|F(\gamma a)+F(\delta b)-F(\gamma a+\delta b)\|_2+\|F(\gamma a +\delta b)-\alpha a\|_2
+\|\delta a-F(\delta b)\|_2+(\max_{j\in {\supp(a)}}\!\gamma_j)\|(F(a),a)\|_{Z_2^b}\]

\[\leq C\left(||\gamma a||_2 +||\delta b||_2\right)+ 2\|T\| \|(a,b)\|_{Z_2^n}+ 2\|T\| \|(a,b)\|_{Z_2^n}
+\max_{j\in {\supp(a)}} |\gamma_j| \ ||a||_2\]
\[\leq C\left(\|T\|\|(a,b)\|_{Z_2^n}+\max_{j\in {\supp(a)}}|\gamma_j| \|a\|_2\right)\,.\]
\end{proof}

\begin{lem}
\label{lem-D}
In the context of the preceding lemmas, for every subset
$A\subset \{1,\ldots,n\}$, $|A|=k> 1$, one has
\[
\left(\frac{1}{k}\sum_{j\in A}\gamma_j^2\right)^\frac{1}{2}\leq
\frac{4\|T\|}{\log k}\,.
\]
\end{lem}
\noindent
\begin{proof}
Let $x=(a,b)\in Z_2^n$ where $a=(a_1,\dots,a_n)$ and $b=(b_1,\dots,b_n)$ are  given by:
\[a_j= \frac{\log\frac{1}{\sqrt{k}}}{\sqrt{k}},\ b_j= \frac{1}{\sqrt{k}}
\mbox{ if } j\in A,\ a_j =b_j=0 \mbox{ otherwise}\,. \]
Then $\|x\|_{Z_2^n}=1$ and therefore
\[\|T\|\geq \|Tx\|_{Z_2^n}\geq \left(\frac{1}{k}\sum_{j\in A}(-\gamma_j\log\sqrt{k}+\delta_j)^2
\right)^{\frac{1}{2}}\geq \log\sqrt{k}
\left(\frac{1}{k}\sum_{j\in A}\gamma_j^2\right)^{\frac{1}{2}}
-\max_j|\delta_j|\,.\]
We now use Lemma~\ref{lem-B} to complete the proof.
\end{proof}

\medskip \noindent
\begin{cor}
\label{cor-E}
In the context of the preceding lemmas,
let $A\subset \{1,\ldots,n\}$, $|A|=k> 1$. Then
\begin{itemize}
\item[a)] There exists a subset
$A'\subset A$, with $|A'|\geq \frac{k}{2}$, such that
\begin{equation}
\label{eq-G}
\max_{j\in A'}|\gamma_j|\leq \frac{4\sqrt{2}\|T\|}{\log k}\,.
\end{equation}
\item[b)] If\/ $\supp(a)\subset A'$, where $A'\subset \{1,\ldots,n\}$,
$|A'|\leq k$ and (\ref{eq-G}) is satisfied, then
\[\left(\sum_{j=1}^n \left(\delta_j a_j-\alpha_j a_j+
\gamma_j a_j\log \frac{|a_j|}{\|a\|_2}\right)^2\right)^{\frac{1}{2}}\leq
C\|T\|\|(a,b)\|_{Z_2^n}\,.\]
\end{itemize}
\end{cor}
\noindent
\begin{proof}
Let $C=\{j\in A; \gamma_j> \frac{4\sqrt{2}\|T\|}{\log k}\}$. To prove a), it is enough
to show that $|C|\le k/2$. If it was not true, one would have
\[\left( \frac{1}{k}  \sum_{j\in A} \gamma_j^2 \right)^{ \frac{1}{2} }
\geq \left( \frac{1}{k} \sum_{j\in C} \gamma_j^2\right)^{\frac{1}{2} }
>\frac{4||T||}{\log k}\,,\]
which contradicts Lemma~\ref{lem-D}.
Part b) follows from Lemma~\ref{lem-C} and the
fact (see\ \cite{KP}) that the norm of the operator
$p_k\,:\,Z_2^k\rightarrow \ell_2^k$, $p_k(a,b)=a$ is equivalent to $\log k$.
\end{proof}
\begin{lem}
\label{lem-F}
Let $A\subset \{1,\ldots,n\}$, $|A|=k>k_0>1$. Then there exists a subset $A''\subset A$,
with $|A''|\geq \frac{\sqrt{k}}{3}$, such that for all $j\in A''$ we have
\[|\gamma_j|\leq \frac{C\|T\|}{(\log k)^2} \mbox{\quad and\quad}
|\delta_j-\alpha_j|\leq \frac{C\|T\|}{\log k}\,.\]
\end{lem}
\noindent
\begin{proof}
Let $A'\subset A$, with $k'=|A'|\geq \frac{k}{2}$ satisfy (\ref{eq-G}). Let
$x=(a,b)\in Z_2^n$ be defined as in the proof of Lemma~\ref{lem-D}, with $A'$
replacing $A$. Then, by Corollary~\ref{cor-E}, b), we have (since $\|x\|=1$)
\[\left(\frac{1}{k'}\sum_{j\in A'}(\delta_j-\alpha_j-\gamma_j\log\sqrt{k'})^2\right)^{\frac{1}{2}}
\leq \frac{C\|T\|}{\log{\sqrt{k'}}}\,.\]
As in the proof of Corollary \ref{cor-E}, there exists a subset $B\subset A'$ with
$k''=|B|\geq \frac{k'}{2}\geq \frac{k}{4}$, such that
\begin{equation}
\label{eq-H}
|\delta_j-\alpha_j-\gamma_j\log{\sqrt{k'}}|\leq \frac{\sqrt{2}C\|T\|}{\log{\sqrt{k'}}}=
\frac{2C_1\|T\|}{\log{k'}}\hbox{ for all $j\in B$. }
\end{equation}

Take now any subset $B'\subset B$ with $|B'|=k'''=\sqrt{k'}\leq \frac{k'}{2}\leq k''$,
and a new vector $x\in Z_2^n$ like the one above, with $B'$ replacing $A'$ (the implicit
assumption that $\sqrt{k'}$ is an integer can easily be dealt with). The same argument
as above provides a subset $A''\subset B'$ with
\[|A''|\geq \frac{k'''}{2}=
\frac{\sqrt{k'}}{2}\geq \frac{\sqrt{k}}{2\sqrt{2}}\,,\]
such that
\begin{equation}
\label{eq-I}
|\delta_j-\alpha_j-\gamma_j\log{\sqrt{k'''}}|\leq \frac{\sqrt{2}C\|T\|}{\log{\sqrt{k'''}}}=
\frac{4C_1\|T\|}{\log{k'}}\hbox{\ for all $j\in A''$.}
\end{equation}
 Adding (\ref{eq-H}) and (\ref{eq-I}) together we get for $j\in A''$:
\[\frac{|\gamma_j|\log k'}{4}=\left|\frac{1}{2}\gamma_j\log k'-\frac{1}{4}\gamma_j\log k'\right|\]
\[=|(-\delta_j+\alpha_j+\gamma_j\log{\sqrt{k'}})+(\delta_j-\alpha_j-\gamma_j\log{\sqrt{k'''}})|\leq
\frac{6C_1\|T\|}{\log k'}\,,\]
and thus :
\begin{equation}
\label{eq-J}
|\gamma_j|\leq \frac{24C_1\|T\|}{(\log k')^2}\leq \frac{C_2\|T\|}{(\log k)^2}\,.
\end{equation}

From (\ref{eq-I}) and (\ref{eq-J}) together we get
\[|\delta_j-\alpha_j|\leq \frac{4C_1\|T\|}{\log{k'}}+|\gamma_j|\log{\sqrt{k'''}}\leq\frac{4C_1\|T\|}{\log{k'}}+\frac{24C_1\|T\|}{(\log k')^2}\cdot\frac{1}{4}\log k'\leq
\frac{C_3\|T\|}{\log k}\hbox{ for $j\in A''$}.\]
\end{proof}

\medskip\noindent
\begin{proof}[Proof of Theorem~\ref{th-A}]

\medskip\noindent
We recall now the construction from \cite{JLS}:  Under the
hypotheses that $Z_2$ has l.u.st., they constructed an
operator $T\,:\,Z_2\rightarrow Z_2$ that provided a counter
example to that hypotheses. In the context of $Z_2^n$, this
construction, instead of providing a counter example, will provide
a lower bound estimate of $\chi_u(Z_2^n)$.

Let the Banach space $Y_n$ with the unconditional basis $\{y_{j,i};\;j=1,\ldots,n,\;
i=1,\ldots,k_j\}$ and the projection $P_n\,:\,Y_n\rightarrow Z_2^n$ be those presented at
the opening. For a fixed $1\leq j\leq n$ a set $A_j\subset \{1,\ldots,k_j\}$ is selected
and an operator $T_{A_j}\,:\,[e_j,f_j]\rightarrow [e_j,f_j]$ is defined by
\[T_{A_j}(x)=P_n\left(\sum_{i\in A_j}y_{j,i}^*(x)y_{j,i}\right)\,,\]
where $\{y_{j,i}^*\}\subset Y_n^*$ are the bi-orthogonal functionals of $\{y_{j,i}\}$. Thus
$T_{A_j}=\sum_{i\in A_j}T_{j,i}$, where $T_{j,i}$ is the rank-1 operator
$T_{j,i}(x)=P_n(y_{j,i}^*(x)y_{j,i})$ ; we may assume that $T_{j,i}\not=0$,
as otherwise $y_{j,i}$ can be dropped from the basis. Being a rank-1 operator, the matrix
representing $T_{j,i}$ is of the form
\[ \left(\begin{array}{ll} a_i&b_i\\\theta_ia_i&\theta_ib_i\end{array}\right)
\mbox{\ (case I), or\ }
\left(\begin{array}{ll} 0&0\\ a_i&b_i\end{array}\right)
\mbox{\ (case II)}\]
(note that we dropped the index $j$ for simplicity). In case II we define $\theta_i$ to be $1$.
We have $\sum_{i=1}^{k_j} \theta_i b_i=1$ (because $\sum_i T_{j,i}f_j=(\sum_i \theta_ib_i) f_j=f_j$)
 and $\sum_{i=1}^{k_j}|b_i|\leq K_n$.

\noindent
Define
\[B_j:=\left\{i\in\{1,\ldots,k_j\};\,|\theta_i|>\frac{1}{2K_n}\right\}\,.\]
We have
\[\sum_{i\in B_j}\theta_ib_i>\frac{1}{2}\ .\]

\noindent
For each $i\in B_j$ at least one of the following four possibilities
holds:
\begin{itemize}
\item[($\iota$)\ \ ]\ \ $\theta_ia_i >\frac{1}{4K_n}|\theta_ib_i|$
\item[($\iota\iota$)\ ] $-\theta_ia_i >\frac{1}{4K_n}|\theta_ib_i|$
\item[($\iota\iota\iota$)] $a_i-\theta_ib_i\geq \frac{1}{2}|\theta_ib_i|$ ($-\theta_ib_i\geq
\frac{1}{2}|\theta_ib_i|$
in case II)
\item[($\iota\nu$)\ ] $\theta_ib_i-a_i\geq \frac{1}{2}|\theta_ib_i|$ ($\theta_ib_i\geq
\frac{1}{2}|\theta_ib_i|$
in case II)
\end{itemize}

\noindent
In fact, if $|\theta_ia_i|>\frac{1}{4K_n}|\theta_ib_i|$ then ($\iota$) or ($\iota\iota$) holds. Otherwise
$|\theta_ia_i|\leq \frac{1}{4K_n}|\theta_ib_i|$. Since $i\in B_j$ this implies in case I
\[\frac{1}{2K_n}|a_i|\leq \frac{1}{4K_n}|\theta_ib_i|\]
which is
\[\left|\frac{a_i}{\theta_i}\right|\leq \frac{1}{2}|b_i|\,. \]
We thus have
\[\left|\frac{a_i}{\theta_i}-b_i\right|\geq |b_i|-\left|\frac{a_i}{\theta_i}\right|
\geq |b_i|-\frac{1}{2}|b_i|=\frac{1}{2}|b_i|\]
which is either ($\iota\iota\iota$) or ($\iota\nu$). The remarks about case II are trivial.

\noindent
It follows that there exists a subset $A_j\subset B_j$ such that one and the same possibility
out of ($\iota$)--($\iota\nu$) holds for all $i\in A_j$ and $\sum_{i\in A_j} \theta_ib_i>\frac{1}{8}$.
the operator $T\,:\,Z_2^n\rightarrow Z_2^n$ that splits and is defined by $T|_{E_j}=T_{A_j}$
satisfies $\|T\|\leq K_n$. Assume that the matrix representing $T_{A_j}$ in the basis $\{e_j,f_j\}$
is, as before,
\[\left(\begin{array}{ll} \alpha_j&\beta_j\\ \gamma_j& \delta_j \end{array}\right)\,.\]
From the above it follows that $\delta_j>\frac{1}{8}$ for all $j$.

\noindent
If ($\iota$) or ($\iota\iota$) is satisfied for all $i\in A_j$ then we have
\begin{equation}
\label{eq-K}
|\gamma_j|>\frac{1}{4K_n}\delta_j>\frac{1}{32K_n}\,.
\end{equation}

\noindent
If ($\iota\iota\iota$) or ($\iota\nu$) is satisfied for all $i\in A_j$ then we have
\begin{equation}
\label{eq-L}
|\alpha_j-\delta_j|>\frac{1}{2}\delta_j>\frac{1}{16}\,.
\end{equation}

\noindent
We can now select a subset $D\subset\{1,\ldots,n\}$, with $|D|\geq \frac{n}{2}$, such that
either (\ref{eq-K}) holds for all $j\in D$ or (\ref{eq-L}) holds for all $j\in D$.
We take $D$ to be the set $A$ of Lemma~\ref{lem-F}. Let $A''\subset D$ be as in Lemma~\ref{lem-F}
and $j\in A''$.

\noindent
If the members of $D$ satisfy (\ref{eq-K}) then we have
\[\frac{1}{32K_n}\leq |\gamma_j|\leq \frac{C\|T\|}{\left(\log\frac{n}{2}\right)^2}\leq
\frac{C_1K_n}{(\log n)^2}\]
and we conclude
\begin{equation}
\label{eq-M}
K_n\geq c\,\log n
\end{equation}
for some positive constant $c$.

\noindent
If the members of $D$ satisfy (\ref{eq-L}) then we have
\[\frac{1}{16}\leq |\delta_j-\alpha_j|\leq \frac{C\|T\|}{\log\frac{n}{2}}\leq
\frac{C_1K_n}{\log n}\]
thus
\begin{equation}
\label{eq-N}
K_n\geq c\,{\log n}\,.
\end{equation}
\vspace{3mm}
\end{proof}

\section{The asymmetry of $Z_2^n$}

{\bf Definition.} We say that  a subgroup $G$ of $GL_k(\R)$ is {\it rich} if it is a compact group
and every operator in $L(\R^k)$ which commutes with all elements
of $G$ is a scalar multiple of the identity on $\R^k$. It is well known that  for any compact
subgroup
$G$ of $GL_k(\R)$, there  exists a compact subgroup $H$ of ${\cal O}_k$ (the orthogonal group)
 such  that $G=\{V^{-1} hV;\, h\in HÖ\}$ where
$V\in L(\R^k)$ is some linear invertible operator.
If $G$ is a rich subgroup of $GL_k(\R)$, we define the {\it measure of symmetry
$s_G(E)$of a $k$-dimensional
normed space $E$ with respect to $G$} by
\[s_G(E):=\int_G ||g||_{E\to E} d\mu(g)\,,\]
where $\mu$ is the normalized Haar measure on $G$.
We define
\[s(E)=\inf\{ s_G(E); G{\rm\ a\ rich \ subgroup}\}\]
Note that
\[1\le s(E)\le S(E):=\inf\{\sup_{g\in G} ||g||_{E\to E}; G{\rm\ a\ rich\ subgroup}\}\,.\]

The quantity $S(E)$, called the {\em asymmetry constant of $E$\/},
was defined originally in \cite{GG} and further discussed in \cite{GL1}. If $S(E)=1$
we say that $E$ has {\it enough symmetries}. this is
the case for spaces with $1$-symmetric basis (take the rich group of permutations and changes
of signs on the  basis). In connection to this quantity we have the following partial result:

\begin{theo}
\label{th-BB} For every rich subgroup $G$ of ${\cal O}_{2n}$ one has
\begin{equation}
\label{eq-BB}
 s_G(Z_2^n)=\int_G || g ||_{ Z_2^n\to Z_2^n}\ d\mu(g)\sim {\log n}\sim\max_{g\in G}
 ||g||_{ Z_2^n\to Z_2^n}
\end{equation}
where $\mu$ denotes the normalized Haar measure on $G$.
\end{theo}
\vskip 1mm\noindent
We use  the following lemma.
\vskip 1mm\noindent
\begin{lem}\label{lem-Q}
Let a couple  $E$ and $F$ of $k$-dimensional normed spaces, an ideal norm $\alpha$ and a
rich subgroup $G$ of
$GL_k(\R)$ be given.
Then for every invertible operator $S\in L(\R^k)$, one has
\begin{equation}
\label{eq-CC}k\le \alpha(S: E\to F)\ \alpha^*(S^{-1}:F\to E) \le k\int_G
||S^{-1}gS||_{E\to E}\ \|g^{-1}||_{F\to F}\  d\mu(g),
\end{equation}
where $\alpha^*$ denotes the conjugate ideal norm of $\alpha$.
In particular, if all elements of $G$ are isometries of $F$, then
\begin{equation}
\label{eq-CC'}k\le \alpha(S: E\to F)\ \alpha^*(S^{-1}:F\to E) \le k\int_G
||S^{-1}gS||_{E\to E}\ \  d\mu(g),
\end{equation}
\vskip 2mm\noindent
\end{lem}
\begin{proof} We first prove  for any $k$-dimensional normed spaces $D$ and $F$ on $\R^n$,
if $I:\R^k\to \R^k$ denotes the identity operator, then
\begin{equation}
\label{eq-DD}
k\le \alpha(I: D\to F)\ \alpha^*(I^{-1}:F\to D) \le k\int_G
||g||_{D\to D}\  ||g^{-1}||_{F\to F}\   d\mu(g).
\end{equation}
By the definition of $\alpha^*$, there exists $U\in L(\R^k)$
such that
\[\alpha(U: D\to F)\alpha^*(I^{-1}:F\to D)={\rm trace}(UI^{-1})={\rm
trace}(U)\,.\]

Let $V=\int_G g^{-1}Ug\ d\mu(g)$ considered as an operator from $D$ to $F$.
Since $G$ is a rich subgroup, one
has $V=cI$ for some $c\in \R$, so that ${\rm trace}(U) ={\rm trace}(V)=kc$. One has
clearly
\[c \alpha(I:D\to F)= \alpha(V:D\to F)\le \int_G\alpha(g^{-1}Ug: D\to
F)\ d\mu(g)\]
\[\le
\ \alpha(U:D\to F)\ \int_G \|g^{-1}\|_{F\to F} \| g\|_{D\to D}\ d\mu(g)\,.\]
It follows that
\[\alpha(I: D\to F)\ \alpha^*(I^{-1}:F\to D)\]
\[ \le \frac{ \alpha(U:D\to F) \alpha^*(I^{-1}:F\to D) }{c}
\int_G\ \|g\|_{D \to D} \|g^{-1}\|_{F\to F}\ \ d\mu(g)\]
\[= k\int_G\ \|g\|_{D\to D}\  \|g^{-1}\|_{F\to F}\ d\mu(g)\,.\]
\noindent
This proves (\ref{eq-DD}). We define a normed space $D$ by setting
$||x||_D=||S^{-1}x||_E$. Then $S^{-1}: D\to E$ and
$S: E\to D$ are isometries, one has
\[\alpha(I: D\to F)=
\alpha(S:E\to F)\hbox{ , } \alpha^*(I^{-1}:F\to D) =
\alpha^*(S^{-1}:F\to E)$$
$$\hbox{ and }   \|g\|_{D\to D}= ||S^{-1}gS||_{E\to E}\,.\]
Together with (\ref{eq-DD}) this implies (\ref{eq-CC}) .
\end{proof}\vspace{2mm}

%{\bf Proof of Theorem \ref{th-BB}.}
\noindent
\begin{proof}[Proof of Theorem~\ref{th-BB}]
We shall use Lemma \ref{lem-Q} with  $S=I:Z_2^n\to \ell_2^{2n}$, where $I$ is
the identity sending the basis $e_j,f_j$ $1\le j\le n$ on the canonical basis of $\ell_2^{2n}$.
We present here
two proofs {\bf A.} and {\bf B.}.
 \vskip 2mm\noindent
{\bf A.} Let $\alpha=\Pi_1$ and $\alpha^*=\gamma_{\infty}$, the ideal norm of factorization
through $L_{\infty}$.
\vskip 2mm

\noindent
{\bf a.} We claim that
\[\Pi_1(I:Z_2^n\to \ell_2^{2n})\ge c{\sqrt n}\log n.\]
In fact, define, for $1\le j\le n$, $x_j\in Z_2^n$ by
\[x_j=\frac{1}{\sqrt n}\left(\log\frac{1}{\sqrt n}\right)e_j + \frac{1}{\sqrt n}\,f_j\]
Then $||I(x_j)||_2\sim\frac{ \log n}{\sqrt n}$ and for every choice of signs
$\eps=(\eps_1,\dots,\eps_n) \in \{-1,1\}^n$, one has
\[||\sum_{j=1}^n \eps_jx_j||_{Z_2^n} =1.\]
Hence
\[\Pi_1(I:Z_2^n\to \ell_2^{2n})\ge\frac{\sum_{j=1}^n ||x_j||_2}{\sup_{\eps}
||\sum_{j=1}^n \eps_jx_j||_{Z_2^n}}\ge c\sqrt{n}\log n.\]

\vskip 2mm\noindent
{\bf b.} Define $P:Z_2^n\to \ell_2^n$ by $P(a,b)=b$. Then $||P||_{Z_2^n\to \ell_2^n}=1$ and
if $W:\ell_2^n\to \ell_2^{2n}$ is the embedding defined by $W(b)=(0,b)$, then
\[\gamma_{\infty} (I^{-1}:\ell_2^{2n}\to Z_2^n)= | |P||_{Z_2^n\to \ell_2^n}
\hskip 2mm \gamma_{\infty} (I^{-1}:\ell_2^{2n}\to Z_2^n)Ê||W||_{\ell_2^n\to \ell_2^{2n}} \]
\[\ge
\gamma_{\infty} (PI^{-1} W:  \ell_2^n\to \ell_2^n) \sim {\sqrt n}\,,\]
because $PI^{-1}W: \ell_2^n\to \ell_2^n$ is actually the identity mapping on $\ell_2^n$.
\vskip 2mm\noindent
{\bf c.} It follows that
\[\Pi_1(I:Z_2^n\to \ell_2^{2n})\gamma_{\infty} (I^{-1}:\ell_2^{2n}\to Z_2^n)\ge cn\log n\,.\]
 \vskip 2mm\noindent
{\bf B.} Let $\alpha$ be the operator norm; then $\alpha^*=i_1$, the integral norm.
It is easy to see that
\[||I||_{Z_2^n\to \ell_2^{2n}}\sim \log n\]
and
\[i_1(I^{-1}:\ell_2^{2n}\to Z_2^n) = i_1(I^{-1}:\ell_2^{2n}\to Z_2^n) ||P||_{Z_2^n\to\ell_2^n}\]
\[\ge
i_1(PI^{-1}: \ell_2^{2n}\to \ell_2^{n})=\sup_{S\not=0} \frac{ {\rm trace} (PI^{-1}S) }
{||S||_{\ell_2^{n} \to \ell_2^{2n}} } \ge n\,.\]
It follows that
\[||I||_{Z_2^n\to \ell_2^{2n} }\hskip 2mm i_1(I^{-1}:\ell_2^{2n}\to Z_2^n)\ge cn\log n\,.\]
\vskip 3mm\noindent
Using either A. or B. together with Lemma~\ref{lem-Q} with $F=\ell_2^{2n}$, $E=Z_2^n$ and $S=I$,
we get that
\[ \int_G || g ||_{ Z_2^n\to Z_2^n}\ d\mu(g)\ge c \log n\,.\]
It follows that
\[ \log n\le \frac{1}{c}\int_G || g ||_{ Z_2^n\to Z_2^n}\ d\mu(g)\le \frac{1}{c}\,\max_{g\in G}
||g ||_{ Z_2^n\to Z_2^n}\le d(Z_2^n, \ell_2^{2n})\sim\log n\]
where $d$ denotes here the Banach-Mazur distance. This proves (\ref{eq-BB}).
\end{proof}

\vskip 3mm\noindent
{\bf Remark.}
One can prove in the same way as {\bf A.a.} in the proof of the last theorem, that if
$I:Z_2^n\to\ell_{\infty}^{2n}$ is the identity mapping sending the basis $(e_j, f_j), 1\le j\le n$,
on the canonical basis of $\ell_{\infty}^{2n}$, then
\[\Pi_1(I:Z_2^n\to \ell_{\infty}^{2n})\ge c{\sqrt n}\log n\,.\]
We have also that
\[\gamma_{\infty}(I^{-1}: \ell_{\infty}^{2n}\to Z_2^n)=
||I^{-1}||_{\ell_{\infty}^{2n}\to Z_2^n}\sim{\sqrt n}\log n\,.\]
It follows from Lemma \ref{lem-Q} that for every rich subgroup $G$ of $GL_{2n}(\R)$, one has
\[cn(\log n)^2\le
\Pi_1(I:Z_2^n\to \ell_{\infty}^{2n})
 \gamma_{\infty}(I^{-1}: \ell_{\infty}^{2n}\to Z_2^n)
 \le  n\int_G || g^{-1}||_{\ell_{\infty}^{2n}\to \ell_{\infty}^{2n}}
||g ||_{Z_2^n\to Z_2^n} d\mu(g)\,.\]
The first inequality is actually an equivalence, in view of the fact that if one takes for $G$
the group
of changes of signs and permutaions of indices of the basis, one has
$ || g^{-1}||_{\ell_{\infty}^{2n}\to \ell_{\infty}^{2n}} =1$ for all $g\in G$ and
\[\int_G ||g ||_{Z_2^n\to Z_2^n} d\mu(g)\le \|I:Z_2^n\rightarrow \ell_{\infty}^{2n}\|
\|I^{-1}:\ell_{\infty}^{2n}\rightarrow Z_2^n\|\leq c\,(\log n)^2\,,\]
as $\|I:Z_2^n\rightarrow \ell_{\infty}^{2n}\|\sim \log n$.

%\begin{thebibliography}{99}
%\end{thebibliography}

\small{
\noindent Y. Gordon: Department of Mathematics, Technion, Israel Institute of Technology,
Haifa 32000, Israel.\newline
Email: gordon@techunix.technion.ac.il\vspace{1mm}

\noindent M. Junge: Department of Mathematics, University of
Illinois, Urbana, IL 61801, USA.
\newline
Email: junge@math.uiuc.edu\vspace{1mm}

\noindent M. Meyer: Universit\'e Paris-Est - Marne-la-Vall\'ee, Laboratoire d'Analyse et de
Math\'ematiques Appliqu\'ees (UMR 8050), Cit\'e Descartes-5, Bd Descartes, Champs-sur-Marne,
77454 Marne-la-Vall\'ee Cedex 2, France.\newline
Email: Mathieu.Meyer@univ-mlv.fr\vspace{1mm}

\noindent S. Reisner: Department of Mathematics, University of
Haifa, Haifa 31905, Israel.
\newline
Email: reisner@math.haifa.ac.il
}

\end{document}